\documentclass[a4paper,12pt]{amsart}
\usepackage[margin=1.3in]{geometry}

\usepackage[utf8]{inputenc}
\usepackage[T1]{fontenc}

\usepackage{amsthm, amssymb, amsmath, amsfonts, mathrsfs}

\usepackage{mathtools}
\usepackage[ocgcolorlinks, colorlinks=true, pdfstartview=FitV, linkcolor=blue, citecolor=blue, urlcolor=blue, pagebackref=false]{hyperref}

\usepackage{microtype}

\newtheorem{theorem}{Theorem}[section]
\newtheorem{lemma}[theorem]{Lemma}

\theoremstyle{remark}
\newtheorem{remark}[theorem]{Remark}

\renewenvironment{proof}[1][Proof]{ {\itshape \noindent {#1.}} }{$\Box$
\medskip}

\numberwithin{equation}{section}
\newcommand{\R}{\mathbb{R}}

\newcommand{\Z}{\mathbb{Z}}

\newcommand{\Pb}{\mathbb{P}}

\newcommand{\E}{\mathbb{E}}

\newcommand{\B}{\mathbf{B}}

\newcommand{\eps}{\varepsilon}
\def\les{\lesssim}

\newcommand{\la}{\langle}
\newcommand{\ra}{\rangle}

\newcommand{\cX}{\mathcal{X}}

\newcommand{\s}{\mathbf{s}}
\newcommand{\bu}{\mathbf{u}}
\newcommand{\x}{\mathbf{x}}
\newcommand{\z}{\mathbf{z}}

\title[Moments of PAM2]{Moments of 2D Parabolic Anderson Model}
\author{Yu Gu, Weijun Xu}

\address[Yu Gu]{Department of Mathematics, Carnegie Mellon University, Pittsburgh, PA, 15213, USA}

\address[Weijun Xu]{Mathematics Institute, University of Warwick, Coventry, CV4 7AL, UK}

\begin{document}
\begin{abstract}
In this note, we use the Feynman-Kac formula to derive a moment representation for the 2D parabolic Anderson model in small time, which is related to the intersection local time of planar Brownian motions.

\bigskip




\end{abstract}
\maketitle

\section{Introduction}

The aim of this note is to study the existence of moments of the solution to the parabolic Anderson model (PAM) in two spatial dimensions, formally given by
\begin{equation}\label{e.pam2}
\partial_{t} u = \frac{1}{2} \Delta u + u \cdot \xi, \qquad (t,x) \in \R^{+} \times \R^{2},
\end{equation}
where $\xi$ is the two dimensional spatial white noise, that is, a generalized Gaussian process with covariance $\E [\xi(x) \xi(y)] = \delta (x-y)$. 

The equation is well-posed in dimension $1$, but the product between $u$ and $\xi$ becomes ill-defined as soon as $d \geq 2$. For $d=2$, the solution $u$ is defined in \cite{gubinelli2015paracontrolled,hairer2014theory,hairer2015simple} as the limit of a sequence of the regularized and renormalized equations. More precisely, fix a symmetric mollifier $\rho: \R^{2} \rightarrow \R^{+}$ with $\rho(x) = \rho(-x)$ and $\int \rho = 1$. Let
\begin{align*}
\rho_\eps(x) = \eps^{-2} \rho(x/\eps), \qquad \xi_{\eps} = \xi \star \rho_\eps, 
\end{align*}
and consider the equation
\begin{equation}\label{e.maineq}
\partial_t u_\eps=\frac12\Delta u_\eps+ (\xi_{\eps} - C_{\eps}) u_{\eps}, 
\end{equation}
for some large constant $C_{\eps}$. Then, for
\begin{equation}\label{e.larco}
C_{\eps} = \frac{1}{\pi} \log \eps^{-1} ,
\end{equation}
the sequence of solutions $\{u_\eps\}$ converges in some weighted H\"older space in probability to a limit $u$ that is independent of the mollification, see e.g. \cite[Theorem 4.1]{hairer2015simple}, and we call this limit $u$ the solution to $2$D PAM. In $d=3$, the mollifier $\rho_\eps(x)=\eps^{-3}\rho(x/\eps)$, and the renormalization constant takes the form $C_{\eps} = c_1\eps^{-1} + c_2 \log \eps^{-1} + O(1)$ \cite{hairer2015multiplicative}. 

So far, most of the results mentioned above focused on the existence of the solution and the convergence of the regularized PDE after renormalization. The statistical properties of $u$ remains a challenge; see \cite{allez2015continuous,cannizzaro2015multidimensional,chouk2016invariance} for some relevant discussions. The goal of this note is to show that the $n$-th moment of the solution $u$ to $2$D PAM exists for small time, and we present a Feynman-Kac formula for $\E[u^n]$. The following is our main result.

\begin{theorem}\label{t.mainth}
	There exists a universal constant $\delta>0$ such that for every $n\in\mathbb{N}$, the $n$-th moment of $u$ exists for $t \in (0, \frac{\delta}{n^{2}})$ with $\E[u(t,x)^n]$ given by \eqref{e.mm}. 
\end{theorem}

\subsection{Heuristic argument}
We first give a heuristic derivation of $\E [u(t,x)]^{n}$ by writing down a representation for $\E[u_{\eps}(t,x)]^{n}$ and passing to the limit formally. 

Suppose $u_\eps(0,x)=u_0(x)$ for some continuous function $u_0$ with $\|u_0\|_{\infty} \leq 1$,  we write the solution to \eqref{e.maineq} by the Feynman-Kac formula
\begin{equation} \label{e.fk}
u_\eps(t,x)=\E_\B\left[u_0(x+B_t)\exp\left( \int_0^t \xi_{\eps} (x + B_s)ds-C_\eps t\right)\right].
\end{equation}
Here, $B=(B_t)_{t\geq 0}$ is a standard planar Brownian motion starting from the origin and independent of the white noise $\xi$, and $C_{\eps}$ is the constant defined in \eqref{e.larco}. We use $\E_\B$ to denote the expectation with respect to $B$. We now proceed to calculating the $n$-th moment of $u_{\eps}(t,x)$. First of all, the covariance function of $\xi_{\eps}$ satisfies
\begin{align*}
\E [\xi_{\eps}(x) \xi_{\eps}(y)] = R_{\eps}(x-y) := \eps^{-2} R \Big( \frac{x-y}{\eps} \Big), 
\end{align*}
where $R = \rho \star \rho$, and $\rho$ is the mollifier used to regularize the noise $\xi$. Next, one raises the expression \eqref{e.fk} to the $n$-th power, and take a further expectation with respect to $\xi_{\eps}$. Since $B$ is independent of $\xi_{\eps}$, one can interchange this expectation with the one with respect to the Brownian motions, and get
\begin{equation}\label{e.fkmm}
\E [u_\eps(t,x)^n]=\E_\B\left[\exp\left(I_n^\eps(t)-nC_\eps t\right)\prod_{k=1}^n u_0(x+B^k_t) \right]. 
\end{equation}
Here, $B^k, k=1,\ldots,n$ are independent Brownian motions, and $\E_{\B}$ denotes the expectation with respect to these $B^{k}$'s. Also, $I_n^\eps(t)$ is given by
\begin{equation} \label{e.intersection}
I_n^\eps(t) = \sum_{k=1}^n \int_0^t\int_0^s R_\eps(B_s^k-B_u^k)duds+\sum_{1\leq i<j\leq n}\int_0^t\int_0^t R_\eps(B_s^i-B_u^j)dsdu,
\end{equation}
where $R_\eps(x)= \eps^{-2} R(x/\eps)$ converges to the Dirac function as $\eps \rightarrow 0$. Note that we do not have the factor $\frac{1}{2}$ in front of the first term since the integration is on the simplex rather than the square $[0,t]^{2}$. It is well known (see for example \cite[Chapter~2]{chen2010random}) that each term in the second term above (when $i \neq j$) converges to the mutual intersection local time of Brownian motion, formally written as $\int_{[0,t]^{2}} \delta(B_{s}^{i} - B_{u}^{j}) ds du$. The first term above (when one has the same Brownian motion in the argument of $R_\eps$) unfortunately does not converge as $\eps \rightarrow 0$, but it does when one subtracts its mean (see \cite{le1992some,varadhan1969appendix,yor1986precisions}). Thus, we define
\begin{equation} \label{e.self_mean}
\nu_\eps(t)=\int_0^t\int_0^s \E_\B[R_\eps(B_s-B_u)]duds, 
\end{equation}
and for every $t \geq 0$, we have
\begin{equation}\label{e.con}
I_n^\eps(t)-n\nu_\eps(t)\to \cX_n(t)
\end{equation}
in probability, where $\cX_n(t)$ is a linear combination of self- and mutual-intersection local times of planar Brownian motions, formally written as
\begin{equation}\label{e.expo}
\begin{aligned}
\cX_n(t)=&\sum_{k=1}^n \int_0^t\int_0^s \Big(\delta(B_s^k-B_u^k)-\E_\B[\delta(B_s^k-B_u^k)] \Big)duds\\
&+\sum_{1\leq i<j\leq n} \int_0^t\int_0^t \delta(B_s^i-B_u^j)dsdu.
\end{aligned}
\end{equation}

It is well known from \cite{le1992some} that $\cX_n(t)$ has exponential moments for small enough $t$ (depending on $n$). In order for the expression \eqref{e.fkmm} to converge, one needs the divergent constant $C_\eps t$ coincides with $\nu_\eps(t)$. A simple calculations shows that this is indeed the case up to an $O(1)$ correction.

\begin{lemma}\label{l.vareq}
There exists constants $\mu_1$ and $\mu_2$ such that 
\[
\nu_\eps(t)-C_\eps t\to t(\mu_1+\mu_2 \log t)
\]
 as $\eps\to 0$.
\end{lemma}

By \eqref{e.con} and Lemma~\ref{l.vareq}, we have 
\[
\begin{aligned}
I_n^\eps(t)-nC_\eps t=&I_n^\eps(t)-n\nu_\eps(t)+n(\nu_\eps(t)-C_\eps t)\\
\to&\cX_n(t)+nt(\mu_1+\mu_2 \log t) 
\end{aligned}
\]
in probability. If the families $\{u_{\eps}(t,x)^{n}\}$ and $\{ e^{I_n^\eps(t)-nC_\eps t}\}$ are both uniformly integrable, then we can pass both sides of  \eqref{e.fkmm} to the limit, and obtain
\begin{equation}\label{e.mm}
\E[u(t,x)^n]=\E_\B\bigg[\exp\big(\cX_n(t)+nt(\mu_1+\mu_2 \log t)\big) \prod_{k=1}^n u_0(x+B^k_t)  \bigg].
\end{equation}
The rest of the note is to show the uniform integrability of $\{u_{\eps}(t,x)^{n}\}$ and $\{ e^{I_n^\eps(t)-nC_\eps t}\}$ for small time $t$, so \eqref{e.mm} does hold. 

\subsection{Discussions}

\begin{remark}
The same argument leads to a similar result in $d=1$, where we choose $C_\eps=0$ and do not have the small time constraint. The renormalized self-intersection local time can be written as 
\[
\begin{aligned}
\int_0^t\int_0^s (\delta(B_s-B_u)-\E_\B[\delta(B_s-B_u)])duds=\frac12\int_\R L_t(x)^2 dx-\frac12\int_\R \E_\B[L_t(x)^2] dx,
\end{aligned}
\]
with $L_t(x)$ denoting the local time of 1D Brownian motion up to $t$.
\end{remark}

\begin{remark}
For $n=1$, the moment formula reads 
\[
\E[u(t,x)]=\E_\B[u_0(x+B_t)e^{\gamma([0,t]^2_<)+t(\mu_1+\mu_2 \log t)}],
\]
with $\gamma([0,t]^2_<)=\int_0^t\int_0^s (\delta(B_s-B_u)-\E_\B[\delta(B_s-B_u)])duds$ representing the self-intersection local time of $B$. It was proved in \cite{le1994exponential} that there exists $t_0>0$ such that 
 \[
 \E_\B[e^{\gamma([0,t]^2_<)}]\left| \begin{array}{ll}
 <\infty & t<t_0,\\
 =\infty & t>t_0.
 \end{array}
 \right.
 \]
 Thus, it is natural to expect that the moments of $u$ do not exist for large $t$, although we do not have a rigorous proof of it. 
\end{remark}

\begin{remark} \label{rm:explosion}
In \cite{allez2015continuous}, the authors defined the $2$D Anderson Hamiltonian $\mathscr{H}=-\Delta+\xi$ on the torus $\mathbb{T}^2=\R^2/\Z^2$ using para-controlled calculus. An interesting application is the exponential tail bounds for the ground state eigenvalue $\Lambda_1$. It was proved in \cite[Proposition 5.4]{allez2015continuous} that there exists $C_1,C_2>0$ such that 
\[
e^{C_1x}\leq \Pb[\Lambda_1\leq x]\leq e^{C_2x}
\]
as $x\to-\infty$. Using the orthonormal eigenvectors of $\mathscr{H}$, denoted by $\{e_n\}$, we write the solution to PAM as \[
u(t,x)=\sum_{n=1}^\infty e^{-\Lambda_n t}\la u_0,e_n\ra e_n(x),
\] 
therefore,
\[
\int_{\mathbb{T}^2} \E[|u(t,x)|^2] dx\leq \E[e^{-2\Lambda_1 t}]\int_{\mathbb{T}^2} |u_0(t,x)|^2 dx.
\]
By the exponential tail bounds on $\Lambda_1$, it is clear the r.h.s. of the above display is only finite for small $t$, which is consistent with our result. 

\end{remark}

\begin{remark}
	In the forthcoming article \cite{dirk}, the authors consider the $2$D PAM with a small noise
	\begin{equation} \label{e.beta}
	\partial_{t} u = \Delta u + \beta u \cdot \xi, \qquad (t,x) \in \R^{+} \times \R^{2}. 
	\end{equation}
	They obtain an explicit chaos expansion of certain polymer measure associated with \eqref{e.beta} for $\beta\ll1$. In particular, this implies that the second moment of $u$ exists for $t \in [0,1], x \in \R^{2}$ and $\beta$ sufficiently small. The restriction of $\beta\ll1$ is equivalent with our small time restriction. Indeed, define
	\begin{align*}
	u_{\beta}(t,x) := u(t/\beta^2, x/\beta), 
	\end{align*}
	one sees that $u_{\beta}$ satisfies \eqref{e.pam2}, hence for $u_\beta(t,x)$ to be square integrable, we need $t/\beta^2\leq 1$, i.e., $t\leq \beta^2\ll1$. 
\end{remark}

\begin{remark}\label{r.de}
A simple calculation shows that the moments of the approximations to $3$D PAM explode as $\eps\to0$, and indicates that the solution to $3$D PAM may not have a moment. To see this, we consider the constant initial condition $u_0 \equiv 1$, so
	\begin{align*}
	\E [u_{\eps}(t,x)] = e^{- C_{\eps} t} \ \E_\B \Big[\exp \Big( {\int_{0}^{t} \int_{0}^{s} R_{\eps}(B_s - B_u) du ds} \Big) \Big],
	\end{align*}
	where $R_\eps(x)=\eps^{-3}R(x/\eps)$.
	
	Since $R(x)$ is continuous and $R(0)>0$, without loss of generality we assume there exists $\delta>0$ such that $R(x)>\delta>0$ for $|x|\leq 2$. Thus, by considering the event that $|B_s| < \eps$ for all  $s\in [0,t]$, we have
	\[
	 \E_\B \Big[\exp \Big( {\int_{0}^{t} \int_{0}^{s} R_{\eps}(B_s - B_u) du ds} \Big) \Big]\geq  \exp\left(\frac{\delta t^2}{2\eps^3}\right)\Pb\big[\sup_{s\in[0,t]}|B_s|<\eps\big].
	 \]
The probability $\Pb[\sup_{s\in[0,t]}|B_s|<\eps]$ is bounded from below by $e^{-c't\eps^{-2}}$ for some $c'>0$ depending on the dimension. When $d=3$, the renormalization constant $C_{\eps}=c_1\eps^{-1}+c_2|\log \eps|+O(1)$. It implies that for any $t>0,x\in\R^3$, we have $\lim_{\eps\to0} \E[u_\eps(t,x)]=\infty$. The same discussion applies to $d=2$, where 
\[
\E[u_\eps(t,x)]\geq  \exp\left(\frac{\delta t^2}{2\eps^2}-\frac{c't}{\eps^2}-C_\eps t\right).
\]
If $t>2c'/\delta$, we also have $\lim_{\eps\to0}\E[u_\eps(t,x)]=\infty$. Since we do not have a proof of $\E[u(t,x)]=\lim_{\eps\to0}\E[u_\eps(t,x)]$ in $d=3$ or $d=2$ for large $t$, we only conjecture that $\E[u(t,x)]=\infty$ in those cases.
\end{remark}

\begin{remark}
When $d=2$, the small time constraint for the existence of moments in our context also appears in \cite[Theorem 4.1]{hu2002chaos}, where the usual product $u\cdot\xi$ is replaced by the Wick product $u\diamond \xi$.
\end{remark}

\begin{remark}
In \cite{gkr}, a similar result is derived for the random Schr\"odinger equation $i\partial_t \phi+\frac12\Delta \phi-\phi\cdot \xi=0$.
\end{remark}

\section{Proof of Lemma \ref{l.vareq} and Theorem \ref{t.mainth}}

We denote $[0,t]^n_<=\{0\leq s_1<\ldots<s_n\leq t\}$, and write $a\les b$ if $a\leq Cb$ with some constant $C$ independent of $\eps$.

\begin{proof}[Proof of Lemma~\ref{l.vareq}]
By scaling property of Brownian motion, we have
\begin{align*}
R_{\eps}(B_s - B_u) = \eps^{-2} R \Big( \frac{B_s - B_u}{\eps} \Big) \stackrel{\text{law}}{=} \eps^{-2} R \big( B_{s/\eps^2} - B_{u/\eps^2} \big). 
\end{align*}
A change of variable $(u/\eps^2,s/\eps^2) \mapsto (u,s)$ then yields
\[
\nu_\eps(t)=\eps^2\int_0^{t/\eps^2}\int_0^{s}\E_\B[R(B_s-B_u)]duds.
\]
Now, $B_s - B_u$ has the normal density $x\mapsto \big( 2\pi (s-u)\big)^{-1} e^{-\frac{|x|^2}{2(s-u)}}$. We then do another change of variable $s-u \mapsto v$, integrate $s$ out, and rescale $v \rightarrow v \eps^{2}$. This leads us to
\begin{align*}
\nu_\eps(t) &= \frac{t}{2\pi} \int_{\R^{2}} R(x)  \bigg( \int_{0}^{t} v^{-1} e^{-\frac{\eps^2 |x|^2}{2v}}dv \bigg) dx - \frac{1}{2\pi}\int_0^t  \bigg(\int_{\R^2} R(x) e^{-\frac{\eps^2|x|^2}{2v}}dx \bigg) dv \\
&:= (\text{i}) - (\text{ii}). 
\end{align*}
Since $R$ integrates to $1$, it is clear that
\[
(\text{ii})\to \frac{t}{2\pi}
\]
as $\eps \rightarrow 0$. As for (i), a substitution of variable $\frac{\eps^2 |x|^2}{2v} \mapsto \lambda$ and then an integration by parts yields
\[
\begin{aligned}
(\text{i}) &= \frac{t}{2\pi}\int_{\R^2}R(x) \left( \int_{\frac{\eps^2|x|^2}{2t}}^\infty \lambda^{-1}e^{-\lambda}d\lambda\right)dx\\
&= \frac{t}{2\pi}\int_{\R^2} R(x)  \left( \int_{\frac{\eps^2|x|^2}{2t}}^\infty e^{-\lambda}\log \lambda d\lambda-e^{-\frac{\eps^2|x|^2}{2t}}\log \Big(\frac{\eps^2|x|^2}{2t}\Big)\right)dx.
\end{aligned}
\]
It is clear from the above expression that as $\eps\to0$, the only divergent part of (i) is from the term $\log (\eps^2)$, and a direct calculation shows
\[
\nu_\eps(t)-\frac{t}{\pi} \cdot |\log \eps| \to \mu_1t+\mu_2 t\log t
\]
for some constant $\mu_1,\mu_2$.
\end{proof}

\begin{proof}[Proof of Theorem~\ref{t.mainth}]
Fix $(t,x)$ and $n$, and recall that 
\begin{equation}\label{e.mmeps}
\E[u_\eps(t,x)^n]=\E_\B\left[\exp(I_n^\eps(t)-nC_\eps t)\prod_{k=1}^n u_0(x+B^k_t)\right], 
\end{equation}
where $\E_{\B}$ is the expectation with respect to independent planar Brownian motions $B^{k}$'s, and $I_n^\eps$ is given by the expression \eqref{e.intersection}. Note that $u_{\eps}(t,x)^n \rightarrow u(t,x)^n$ in probability, and that by \eqref{e.con} and Lemma \ref{l.vareq}, we have
\begin{align*}
 I_n^\eps(t)-nC_\eps t\to\cX_n(t)+nt(\mu_1+\mu_2 \log t)
\end{align*}
in probability. Thus, in view of \eqref{e.mmeps}, it suffices to show the uniform integrability of $u_{\eps}(t,x)^n$ and $\exp(I_n^\eps(t)-nC_\eps t)\prod_{k=1}^n u_0(x+B^k_t)$. This allows us to pass both sides of \eqref{e.mmeps} to the limit and conclude Theorem \ref{t.mainth}. 

To prove the uniform integrability, we bound the second moment of these two objects:
\begin{align*}
\E\big[|u_\eps(t,x)|^{2n}\big] \les \E_\B \big[e^{I_{2n}^\eps(t)-2nC_\eps t}\big] \les \E_\B\big[ e^{I_{2n}^\eps(t)-2n\nu_\eps(t)}\big],
\end{align*}
and 
\begin{align*}
\E_\B\Big[\big|e^{I_n^\eps(t)}e^{-nC_\eps t}\prod_{k=1}^n u_0(x+B^k_t)\big|^2\Big]\les \E_\B\big[e^{2I_n^\eps(t)-2nC_\eps t}\big] \les \E_\B\big[ e^{2(I_n^\eps(t)-n\nu_\eps(t))}\big], 
\end{align*}
where we have used $\|u_0\|_\infty\leq 1$. Thus, it suffices to show that for every $n$ and $\theta$, there exists $t_0$ small enough such that $\E_\B[e^{ \theta (I_n^\eps(t)-n\nu_\eps(t))}]$ is uniformly bounded in $\eps$ for all $t<t_0$. To see this, using H\"older's inequality, we get
\begin{align*}
\E_\B \big[e^{ \theta (I_n^\eps(t)-n\nu_\eps(t))} \big] \leq \prod_{k=1}^{n} \Big[ \E_\B  e^{\theta N [\beta_{\eps}^{k}([0,t]^{2}_{<})-\E_\B\beta_{\eps}^{k}([0,t]^{2}_{<})]} \Big]^{\frac{1}{N}} \prod_{1\leq i<j \leq n}  \Big( \E_{\B} e^{\theta N \alpha_{\eps}^{i,j}([0,t]^{2})} \Big)^{\frac{1}{N}}, 
\end{align*}
where $N=\frac{n(n+1)}{2}$, and we have used the notations
\begin{align*}
\beta_\eps^k([0,t]^2_<)=\int_0^t \int_0^s R_\eps(B^k_s-B^k_u)duds, \  \ \alpha_\eps^{i,j}([0,t]^2)=\int_0^t\int_0^t R_\eps(B^i_s-B^j_u)dsdu.
\end{align*}
 By change of variables and the scaling property of the Brownian motion, we have 
\begin{align*}
\beta_\eps^k([0,t]^2_<) \stackrel{\text{law}}{=}t \beta_{\eps/\sqrt{t}}^k([0,1]_<^2), \qquad  \alpha_\eps^{i,j}([0,t]^2)\stackrel{\text{law}}{=}t\alpha_{\eps/\sqrt{t}}^{i,j}([0,1]^2). 
\end{align*}
Then, Lemma~\ref{l.uniform} implies that there exists $\lambda, C > 0$ such that 
\begin{align*}
t < \frac{\lambda}{\theta N} \  \ \Rightarrow \  \ \sup_{\eps\in(0,1)} \E_\B \big[e^{ \theta (I_n^\eps(t)-n\nu_\eps(t))} \big] \leq C. 
\end{align*}
This completes the proof. 
\end{proof}

\appendix
\section{Exponential moments of intersection local time of planar Brownian motions}

Recall that $R_\eps(x)=\eps^{-2}R(\frac{x}{\eps})$, we define 
\[
\alpha_\eps(A)=\int_{A} R_\eps(B_s^1-B_u^2)dsdu, \  \ \beta_\eps(A)=\int_{A} R_\eps(B_s-B_u)dsdu
\] 
for any set $A\subset \R_+^2$, and 
\[
X_\eps=\beta_\eps([0,1]^2_<)-\E_\B[\beta_\eps([0,1]^2_<)], \  \  Y_\eps=\alpha_\eps([0,1]^2).
\]
\begin{lemma}\label{l.uniform}
There exists universal constants $\lambda,C>0$ such that 
\[
\sup_{\eps\in(0,1)} \left( \E_\B[ e^{\lambda X_\eps}]+\E_{\B}[e^{\lambda Y_\eps}]\right) \leq C.
\]
\end{lemma}
The above result is standard. The case $\eps=0$, i.e., the exponential integrability of intersection local time, was addressed in the classical work \cite{le1994exponential}. We could not find a direct reference for $\eps>0$, though the proof follows essentially in the same line as the case of $\eps=0$. For the convenience of the reader, we present the details here.

\begin{proof}
We consider $Y_\eps$ first. Since $R=\rho\star \rho$, we can write 
\[
Y_\eps=\int_{[0,1]^2}\int_{\R^2} \rho_\eps(B_s^1-x)\rho_\eps(B_u^2-x)dx dsdu,
\]
with $\rho_\eps(x)=\eps^{-2}\rho(x/\eps)$. For any $n\in \mathbb{N}$,
\[
\begin{aligned}
\E_\B[Y_\eps^n]=&\int_{\R^{2n}}\left( \int_{[0,1]^{2n}} \E_\B\left[\prod_{k=1}^n \rho_\eps(B_{s_k}^1-x_k)\rho_\eps(B_{u_k}^2-x_k)\right]d\s d\bu \right)d\x\\
=&\int_{\R^{2n}} \left(\int_{[0,1]^n} \E_\B\left[\prod_{k=1}^n \rho_\eps(B_{s_k}-x_k)\right]d\s\right)^2 d\x.
\end{aligned}
\]
By \cite[(2.2.11)]{chen2010random}, we have 
\[
\E_\B[Y_\eps^n]= \int_{\R^{2n}}  \left( \int_{\R^{2n}}\prod_{k=1}^n \rho_\eps(z_k-x_k)\sum_{\sigma} \int_{[0,1]^n_<}  \prod_{k=1}^n p_{s_k-s_{k-1}}(z_{\sigma(k)}-z_{\sigma(k-1)})d\s d\z\right)^2d\x,
\]
where $p_t(x)$ is the density of $N(0,t)$, $[0,t]^n_<=\{0\leq s_1<\ldots<s_n\leq t\}$, and $\sum_\sigma$ denotes the summation over all permutations over $\{1,\ldots,n\}$. If we denote 
\[
h(z_1,\ldots,z_n)=\sum_{\sigma} \int_{[0,1]^n_<}  \prod_{k=1}^n p_{s_k-s_{k-1}}(z_{\sigma(k)}-z_{\sigma(k-1)})d\s, \  \ Q_\eps(z_1,\ldots,z_n)=\prod_{k=1}^n \rho_\eps(z_k),
\] then $\E_\B[Y_\eps^n]$ equals to 
\begin{equation}\label{e.lemm}
\begin{aligned}
\int_{\R^{2n}}  |Q_\eps\star h(x_1,\ldots,x_n)|^2 d\x \leq& \left(\int_{\R^{2n}} Q_\eps(x_1,\ldots,x_n)d\x\right)^2 \int_{\R^{2n}}| h(x_1,\ldots,x_n)|^2 d\x\\
=& \int_{\R^{2n}}| h(x_1,\ldots,x_n)|^2 d\x= \E_\B[ \alpha([0,1]^2)^n],
\end{aligned}
\end{equation}
where $\alpha([0,1]^2)$ is the mutual-intersection local time formally written as 
\[
\alpha([0,1]^2)=\int_0^1\int_0^1 \delta(B^1_s-B^2_u)dsdu,
\] 
and we used the Le Gall's moment formula in the second line of \eqref{e.lemm}. By \cite{le1994exponential}, we have
\[
\E_\B[\exp(\mu\alpha([0,1]^2))]<C
\]
 for some $\mu>0$, hence we only need to choose $\lambda=\mu$
to get \[
 \E_\B[e^{\lambda Y_\eps}]=\sum_{n=0}^\infty \frac{\lambda^n\E_\B[Y_\eps^n]}{n!}\leq \sum_{n=0}^\infty \frac{\lambda^n }{n!} \E_\B[|\alpha([0,1]^2)|^n]= \E_\B[e^{\mu\alpha([0,1]^2)}]<\infty.
\]

Next, we consider $X_\eps$. We define the triangle approximation of $\{(u,s):0\leq u<s\leq 1\}$:
\[
A_l^k= \Big[\frac{2l}{2^{k+1}},\frac{2l+1}{2^{k+1}}\Big) \times \Big[ \frac{2l+1}{2^{k+1}},\frac{2l+2}{2^{k+1}} \Big), \  \ l=0,1,\ldots,2^{k-1}, k=0,1,\ldots.
\] 
We will use the following three properties:

(i) Fix any $k$, $\{\beta_\eps(A_l^k)\}_{l=0,\ldots,2^k-1}$ are i.i.d. random variables.

(ii) $\beta_\eps(A_l^k)\stackrel{\text{law}}{=}2^{-(k+1)}\beta_{\eps 2^{(k+1)/2}}([0,1]\times [1,2])\stackrel{\text{law}}{=}2^{-(k+1)}\alpha_{\eps 2^{(k+1)/2}}([0,1]^2)$

(iii)
$\sup_{\eps>0}\E_\B[e^{\lambda \alpha_{\eps}([0,1]^2)}] \leq C$ for some $\lambda,C>0$.

By (iii) and a Taylor expansion, there exists $C>0$ such that for sufficiently small $\lambda$
\begin{equation}\label{e.emm2}
\sup_{\eps>0}\E_\B[e^{\lambda(\alpha_\eps([0,1]^2)-\E_\B[\alpha_\eps([0,1]^2)])}]\leq e^{ C\lambda^2}.
\end{equation}
We fix the constants $\lambda,C$ from now on, and write 
\[
X_\eps=\sum_{k=0}^\infty \sum_{l=0}^{2^k-1} (\beta_\eps(A_l^k)-\E_\B[\beta_\eps(A_l^k)]).
\]
Fix $a\in (0,1)$ and define a sequence of constants 
\[
b_1=2\lambda, \ \ b_N=2\lambda\prod_{j=2}^N (1-2^{-a(j-1)}), N=2,3,\ldots,
\] 
we have 
\[
\begin{aligned}
&\E_\B\exp\left[b_N\sum_{k=0}^N \sum_{l=0}^{2^k-1} (\beta_\eps(A_l^k)-\E_\B\beta_\eps(A_l^k))\right]\\
\leq &\left(\E_\B \exp\left[b_{N-1}\sum_{k=0}^{N-1} \sum_{l=0}^{2^k-1} (\beta_\eps(A_l^k)-\E_\B\beta_\eps(A_l^k))\right]\right)^{1-2^{-a(N-1)}}\\\
&\times \left(\E_\B\exp\left[2^{a(N-1)}b_N\sum_{l=0}^{2^N-1}(\beta_\eps(A_l^N)-\E_\B\beta_\eps(A_l^N))\right]\right)^{2^{-a(N-1)}}\\
\leq &\E_\B\exp\left[b_{N-1}\sum_{k=0}^{N-1} \sum_{l=0}^{2^k-1} (\beta_\eps(A_l^k)-\E_\B\beta_\eps(A_l^k))\right]\\
&\times\left(\E_\B\exp\left[2^{a(N-1)}b_N(\beta_\eps(A_0^N)-\E_\B\beta_\eps(A_0^N))\right]\right)^{2^{N-a(N-1)}}.
\end{aligned}
\]
Since $\beta_\eps(A_0^N)\stackrel{\text{law}}{=}2^{-(N+1)}\alpha_{\eps 2^{(N+1)/2}} ([0,1]^2)$, we have 
\[
\begin{aligned}
&\E_\B\exp\left[2^{a(N-1)}b_N(\beta_\eps(A_0^N)-\E_\B\beta_\eps(A_0^N))\right]\\
=&\E_\B\exp\left[2^{a(N-1)}b_N2^{-(N+1)}(\alpha_{\eps 2^{(N+1)/2}} ([0,1]^2)-\E_\B\alpha_{\eps 2^{(N+1)/2}} ([0,1]^2))\right].
\end{aligned}
\]
Using the fact that $2^{a(N-1)}b_N2^{-(N+1)}<\lambda$ and \eqref{e.emm2}, we derive for all $\eps>0$ that 
\[
\E_\B\exp\left[2^{a(N-1)}b_N(\beta_\eps(A_0^N)-\E_\B\beta_\eps(A_0^N))\right]\leq e^{Cb_N^2 2^{-2N+2a(N-1)}},
\]
so there exists $C'>0$ such that 
\[
\begin{aligned}
&\E_\B\exp\left[b_N\sum_{k=0}^N \sum_{l=0}^{2^k-1} (\beta_\eps(A_l^k)-\E_\B\beta_\eps(A_l^k))\right]\\
\leq &\E_\B\exp\left[b_{N-1}\sum_{k=0}^{N-1} \sum_{l=0}^{2^k-1} (\beta_\eps(A_l^k)-\E_\B\beta_\eps(A_l^k))\right]e^{C'2^{(a-1)N}}.
\end{aligned}
\]
Iterating the above inequality, we get 
\[
\E_\B\exp\left[b_N\sum_{k=0}^N \sum_{l=0}^{2^k-1} (\beta_\eps(A_l^k)-\E_\B\beta_\eps(A_l^k))\right] \leq \exp(C'(1-2^{a-1})^{-1})
\]
Since $b_N\to b_\infty$ for some $b_\infty>0$, we have 
\[
\E_\B[\exp(b_\infty X_\eps)] \leq \exp(C'(1-2^{a-1})^{-1}),
\]
which completes the proof.
\end{proof}

\subsection*{Acknowledgments} We thank the anonymous referees for a very careful reading of our paper and helpful suggestions and comments. 
We thank Dirk Erhard and Nikolaos Zygouras for stimulating discussions and for showing us the argument in Remark~\ref{r.de}. YG is partially supported by the NSF through DMS-1613301. WX is
supported by EPSRC through the research fellowship EP/N021568/1.


\end{document}